\documentclass[11pt]{amsart}
\usepackage{amssymb,amsfonts}
\input xy
\xyoption{all}

\newcommand{\Comp}{\mathbf{Comp}}
\newcommand{\Tych}{\mathbf{Tych}}
\newcommand{\w}{\omega}
\newcommand{\pr}{\mathrm{pr}}
\newcommand{\supp}{\mathrm{supp}}

\renewcommand{\int}{\operatorname{Int}}
\newcommand{\cl}{\operatorname{cl}}
\newcommand{\id}{\mathrm{id}}

\newcommand{\cDelta}{{\Delta\kern-7pt{\mbox{\tiny $\Delta$}}}}
\newcommand{\U}{\mathcal U}

\newcommand{\IN}{\mathbb N}

\newcommand{\mathscr}{\mathit}

\newcommand{\invlim}{\lim}
\newcommand{\V}{\mathcal V}

\newtheorem{theorem}{Theorem}[section]
\newtheorem{lemma}[theorem]{Lemma}
\newtheorem{claim}[theorem]{Claim}
\newtheorem{corollary}[theorem]{Corollary}
\newtheorem{example}[theorem]{Example}
\newtheorem{problem}[theorem]{Problem}
\newtheorem{proposition}[theorem]{Proposition}
\theoremstyle{definition}
\newtheorem{definition}[theorem]{Definition}
\newtheorem{remark}[theorem]{Remark}

\title[Functors preserving skeletal maps]{On functors preserving skeletal maps and skeletally generated compacta}
\author{Taras Banakh, Andrzej Kucharski, and Marta Martynenko}
\address{T.~Banakh: Faculty of Mechanics and Mathematics, Ivan Franko National University of Lviv (Ukraine) and
Instytut Matematyki, Jan Kochanowski University, Kielce (Poland)}
\email{t.o.banakh@gmail.com}
\address{A.Kucharski: Institute of Mathematics, University of Silesia, ul. Bankowa 14, 40-007 Katowice (Poland)}
\email{akuchar@math.us.edu.pl}
\address{M.~Martynenko: Faculty of Mechanics and Mathematics, Ivan Franko National University of Lviv (Ukraine)}
\email{martamartynenko@ukr.net}
\subjclass{18B30, 54B30, 54C10, 54B35, 54D30}
\keywords{Skeletal map, functor, skeletally generated compact space}
\begin{document}

\begin{abstract} A map $f:X\to Y$ between topological spaces is {\em skeletal} if  the preimage $f^{-1}(A)$ of each nowhere dense subset $A\subset Y$ is nowhere dense in $X$. We prove that a normal functor $F:\Comp\to\Comp$ is skeletal (which means that $F$ preserves skeletal epimorphisms) if and only if for any open surjective map $f:X\to Y$ between metrizable zero-dimensional compacta with two-element non-degeneracy set $N^f=\{x\in X:|f^{-1}(f(x))|>1\}$ the map $Ff:FX\to FY$ is skeletal. This characterization implies that each open normal functor is skeletal. The converse is not true even for normal functors of finite degree. The other main result of the paper says that each normal functor $F:\Comp\to\Comp$ preserves the class of skeletally generated compacta. This contrasts with the known \v S\v cepin's result saying that a normal functor is open if and only if it preserves the class of openly generated compacta.
\end{abstract}
\maketitle

\section{Introduction}

In this paper we address the problem of preservation of skeletal maps and skeletally generated compacta by normal functors in the category $\Comp$ of compact Hausdorff spaces and their continuous maps. In the sequel all spaces are Hausdorff and all maps are continuous. A {\em compactum} is a compact Hausdorff space.

A map $f:X\to Y$ between topological spaces is called {\em skeletal\/} if  for each nowhere dense subset $A\subset Y$ the preimage $f^{-1}(A)$ is nowhere dense in $X$, see \cite{mr}. It is easy to see that each open map is skeletal while the converse is not true.

%Given a map $f:X\to Y$ between topological spaces consider its {\em lower and upper non-degeneracy sets}
%$$N_f=\{y\in Y:|f^{-1}(y)|>1\}\mbox{ and }N^f=\{x\in X:|f^{-1}(f(x))|>1\}=f^{-1}(N_f)$$ and its {\em lower and upper degeneracy sets}
% $D_f=Y\setminus N_f$ and $D^f=X\setminus N^f.$

To formulate our main results, we need to recall some definitions from the topological theory of functors, see \cite{TZ}. A functor $F:\Comp\to\Comp$ in the category $\Comp$ is defined to be
\begin{itemize}
\item {\em monomorphic} (resp. {\em epimorphic}) if for each injective (resp. surjective) map $f:X\to Y$ between compacta the map $Ff:FX\to FY$ is injective (resp. surjective);
\item ({\em finitely}) {\em open} if for any open surjection $f:X\to Y$ between (finite) compacta the map $Ff:FX\to FY$ is open;
\item ({\em finitely}) {\em skeletal} if for any open surjection $f:X\to Y$ between (finite) compacta the map $Ff:FX\to FY$ is skeletal;
\item {\em weight-preserving} if $w(FX)\le w(X)$ for each infinite compactum $X$.
\end{itemize}
Here $w(X)$ denotes the {\em weight} (i.e. the smallest cardinality of a base of the topology) of $X$.
If a functor $F:\Comp\to\Comp$ is monomorphic, then for each closed subspace $X$ of a compact space $Y$  the map $Fi:FX\to FY$ induced by the inclusion $i:X\to Y$ is a topological embedding, which allows us to identify the space $FX$ with the subspace $Fi(FX)$ of $FY$.

Next we recall and define several properties of functors related to the bicommutativity. Let $\mathcal D$ be a commutative square diagram
$$\xymatrix{
\tilde X\ar_{p_X}[d]\ar^{\tilde f}[r]&\tilde Y\ar^{p_Y}[d]\\
X\ar_{f}[r]&Y
}$$consisting of continuous maps between compact spaces.

The diagram $\mathcal D$ is called {\em bicommutative} if $\tilde f\big(p_X^{-1}(x)\big)=p_Y^{-1}\big(f(x)\big)$ for all $x\in X$, see \cite[\S3.IV]{Kur1}, \cite[\S2.1]{Sh}.

We say that a functor $F:\Comp\to\Comp$
\begin{itemize}
\item is ({\em finitely}) {\em bicommutative} if $F$ preserves the bicommutativity of square diagrams $\mathcal D$ consisting of surjective maps $f,\tilde f,p_X,p_Y$ and (finite) compacta  $X,Y,\tilde X,\tilde Y$;
\item {\em preserves} ({\em finite}) {\em preimages} if $F$ preserves the bicommutativity of square diagrams $\mathcal D$ with injective maps $f,\tilde f$ (and finite space $X$);
\item {\em preserves} ({\em finite}) {\em 1-preimages} if $F$ preserves the bicommutativity of square diagrams $\mathcal D$ with injective maps $f,\tilde f$, bijective map $p_X$ (and finite space $X$).
\end{itemize}

It is clear that each bicommutative functor is finitely bicommutative. The converse is true for normal functors with finite supports, see
Proposition 2.10.1 of \cite{TZ}.
It is easy to see that a monomorphic functor $F:\Comp\to\Comp$ preserves [finite] (1-)preimages  if and only if for any map $f:X\to Y$ between compact spaces and a [finite] closed subset $Z\subset Y$ (such that $f^{-1}(z)$ is a singleton for every $z\in Z$) we get $(Ff)^{-1}(FZ)=F\big(f^{-1}(Z)\big)$.
\smallskip

A functor $F:\Comp\to\Comp$ will be called {\em mec} if $F$ is monomorphic, epimorphic, and continuous.
A mec functor that preserves finite 1-preimages will be called a {\em 1-mec} functor. A 1-mec functor that preserves weight of infinite compacta will be called a {\em 1-mecw} functor. The class of 1-mecw functors includes all normal functors in the sense of \v S\v cepin \cite{Sh},
\cite[\S2.3]{TZ} (let us recall that a functor $F:\Comp\to\Comp$ is {\em weakly normal} if it is monomorphic, epimorphic, continuous and preserves intersections, the empty set, the singleton, and the weight of infinite compacta; $F$ is {\em normal} if it is weakly normal and preserves preimages).
\smallskip

Our primary aim is to characterize skeletal functors among 1-mec functors. For a topological space $Z$
consider the open map $\mathscr 2_Z:Z\oplus 2\to Z\oplus 1$ defined by
$$\mathscr 2_Z:z\mapsto\begin{cases} z&\mbox{if $z\in Z$,}\\
0&\mbox{if $z\in 2$}.\end{cases}
$$
Here for a natural number $n$ by $Z\oplus n$ we denote the topological sum of $Z$ and the discrete space $n=\{0,\dots,n-1\}$.

\begin{theorem}\label{s-map} A 1-mec (resp. 1-mecw) functor $F:\Comp\to\Comp$ is skeletal if and only if  for each zero-dimensional compact (metrizable) space $Z$ the map $F\mathscr 2_Z:F(Z\oplus 2)\to F(Z\oplus 1)$ is skeletal.
\end{theorem}

Since each open map is skeletal, the preceding theorem implies:

\begin{corollary}\label{op->skel} Each open 1-mec functor  $F:\Comp\to\Comp$ is skeletal.
\end{corollary}

Examples~\ref{ex:TZ}--\ref{e:PDelta} presented in Section~\ref{s:eop} show that Corollary~\ref{op->skel} cannot be reversed.

Next, we discuss the interplay between the skeletality and the (finite) bicommutativity of functors.

\begin{theorem}\label{t1.5n} A 1-mecw functor $F:\Comp\to\Comp$ is skeletal if it is finitely bicommutative and finitely skeletal.
\end{theorem}

This criterion should be compared with the following characterization of open functors due to \v S\v cepin, see Propositions 3.18 and 3.19 of \cite{Sh}.

\begin{theorem}[\v S\v cepin]\label{scepin} A normal functor $F:\Comp\to\Comp$ is open if and only if $F$ is bicommutative and finitely open.
\end{theorem}
\smallskip

Now let us discuss the problem of preservation of skeletally generated compacta by normal functors. Following \cite{Val} we say that a compact Hausdorff space $X$ is {\em skeletally generated} if $X$ is homeomorphic to the limit of an inverse continuous $\w$-spectrum $\mathcal S=\{X_\alpha,\pi_\alpha^\beta,\Sigma\}$ consisting of metrizable compacta and surjective skeletal bonding projections $\pi_\alpha^\beta:X_\beta\to X_\alpha$.

According to \cite{kp8} or \cite{Val}, a compact Hausdorff space $X$ is skeletally generated if and only if the first player has a winning strategy in the following open-open game. The player I starts the game selecting a non-empty open set $V_0$ and the player II responds with a non-empty open set $W_0\subset V_0$. At the $n$-th inning the player I chooses a non-empty open set $V_n$ and player II responds with a non-empty open set $W_n\subset V_n$. At the end of the game the player I is declared the winner if the union $\bigcup_{n\in\w}W_n$ is dense in $X$. Otherwise the player II wins the game.

The class of skeletally generated compacta contains all openly generated compacta and all continuous images of openly generated compacta, see \cite{Val}. In particular, each dyadic compactum is skeletally generated. Skeletally generated compacta share some properties of dyadic compacta. In particular, each skeletally generated compactum has countable cellularity, see \cite{dkz} or \cite{kp8}.

It is known that in general, normal functors do not preserve openly generated compacta. In fact, a normal functor $F:\Comp\to\Comp$ preserves the class of openly generated spaces if and only if $F$ is open, see \cite[\S4.1]{Sh}.
This contrasts with the following theorem.

\begin{theorem}\label{s-space} Each 1-mecw functor $F:\Comp\to\Comp$ preserves the class of skeletally generated compacta.
\end{theorem}

For preimage preserving mecw-functors $F$ with $F1\ne F2$ this theorem can be improved as follows.
Below we identify a natural number $n$ with the discrete space $n=\{0,\dots,n-1\}$.

\begin{theorem}\label{ps-space} Let $F:\Comp\to\Comp$ be a preimage preserving mecw-functor with $F1\ne F2$. A compact Hausdorff space $X$ is skeletally generated if and only if the space $FX$ is skeletally generated.
\end{theorem}

\begin{remark} Among the properties composing the definition of 1-mec functor the least studied is the property of preservation of finite 1-preimages. It is clear that a functor $F:\Comp\to\Comp$ preserves (finite) 1-preimages if it preserves (finite) preimages. On the other hand, the functor of superextension $\lambda$ and the functor of order-preserving functionals $O$ preserve 1-preimages but fail to preserve finite preimages, see \cite{KR}, \cite[2.3.2, 2.3.6]{TZ}. A simple example of a mec functor that does not preserve finite 1-preimages is $Pr^3$, the functor of the third projective power, see \cite[2.5.3]{TZ}. Another example of such a functor is $E$, the functor of non-expanding functionals, see \cite{KR}. By Theorem 1 of \cite{KR}, a continuous monomorphic functor $F:\Comp\to\Comp$ preserves 1-preimages if and only if its Chigogidze extension $F_\beta:\Tych\to\Tych$ preserves embeddings of Tychonoff spaces.
\end{remark}

Theorems~\ref{s-map}, \ref{t1.5n}, \ref{s-space} and \ref{ps-space} will be proved in Sections~\ref{s:s-map}---\ref{s:ps-space} after some preliminary work done in Sections~\ref{s:sm}--\ref{s:fsm}. Several examples of skeletal and non-skeletal functors will be given in Section~\ref{s:eop}. In that section we also pose some open problems related to skeletal functors.

\section{Skeletal maps and skeletal squares}\label{s:sm}

In this section we recall the necessary information on skeletal maps between compact spaces. First we introduce the necessary definitions.

\begin{definition} A map $f:X\to Y$ between two topological spaces is defined to be
\begin{itemize}
\item {\em skeletal at a point $x\in X$} if for each neighborhood $U\subset X$ of $x$ the closure $\cl_Y(f(U))$ of $f(U)$ has non-empty interior in $Y$;
\item {\em skeletal at a subset} $A\subset X$ if $f$ is skeletal at each point $x\in A$;
\smallskip

\item {\em open at a point $x\in X$} if for each neighborhood $U\subset X$ of $x$ the image $f(U)$ is a neighborhood of $f(x)$;
\item {\em open at a subset $A\subset X$} if $f$ is open at each point $x\in A$;
\item {\em densely open} if $f$ is open at some dense subset $A\subset X$.
\end{itemize}
\end{definition}

It is easy to see that each densely open map is skeletal. The converse is true for skeletal maps between metrizable compacta:

\begin{theorem}\label{l1c} A map $f:X\to Y$ between compact metrizable spaces is skeletal if and only if it is open at a dense $G_\delta$-subset of $X$.
\end{theorem}

This theorem has been proved in \cite{BKM}. The metrizability of $X$ is essential in this theorem as shown by the projection $\pr:A\to[0,1]$ of the Aleksandrov ``two arrows'' space $A$ onto the interval, see \cite[3.10.C]{En}. This projection is skeletal but is open at no point of $A$.

A characterization of skeletal maps between non-metrizable compact spaces was given in \cite{BKM} in terms of morphisms of inverse $\w$-spectra with skeletal limit or bonding squares.

\begin{definition} Let $\mathcal D$ denote a commutative diagram $$\xymatrix{
\tilde X\ar[r]^{\tilde f}\ar[d]_{p_X}&\tilde Y\ar[d]^{p_Y}\\
X\ar[r]_{f}&Y
}$$ consisting of maps between compact spaces.
The square $\mathcal D$ is defined to be
\begin{itemize}
\item {\em open at a point $x\in X$} if for each neighborhood $U\subset X$ of $x$ the point $f(x)$ has a neighborhood $V\subset Y$ such that $V\subset f(U)$ and $p^{-1}_Y(V)\subset\tilde f(p_X^{-1}(U))$;
\item {\em open at a subset $A\subset X$} if $\mathcal D$ is open at each point $x\in A$;
\item {\em densely open} if it is open at some dense subset $A\subset X$;
\smallskip

\item {\em skeletal at a point $x\in X$} if for each neighborhood $U\subset X$ of $x$ there is a non-empty open set $V\subset Y$ such that $V\subset f(U)$ and $p_Y^{-1}(V)\subset \tilde f(p_X^{-1}(U))$;
\item {\em skeletal at a subset} $A\subset X$ if $\mathcal D$ is skeletal at each point $x\in A$;
\item {\em skeletal} if $\mathcal D$ is skeletal at $X$.
\end{itemize}
\end{definition}

\begin{remark}\label{rem1a} Let $A$ be a subset of the space $X$ in the diagram $\mathcal D$.
\begin{enumerate}
\item If the diagram $\mathcal D$ is skeletal (resp. open) at $A$, then the map $f$ is skeletal (resp. open) at $A$;
\item The diagram $\mathcal D$ is skeletal (resp. open) at $A$ if it is bicommutative and the map $f$ is skeletal (resp. open) at $A$;
\item If the diagram $\mathcal D$ is open at $A$, then it is {\em bicommutative at $A$} in the sense that $\tilde f(p_X^{-1}(x))=p_Y^{-1}(f(x))$ for all points $x\in A$.
\end{enumerate}
\end{remark}

\begin{remark}\label{rem1} A map $f:X\to Y$ is skeletal (resp. open) at a subset $A\subset X$ if and only if the square  $$\xymatrix{
X\ar[r]^{f}\ar[d]_{\id_X}&Y\ar[d]^{\id_Y}\\
X\ar[r]_{f}&Y
}$$is skeletal (resp. open) at the subset $A$.
\end{remark}

It is easy to see that each densely open square is skeletal. The converse is true in the metrizable case. The following lemma proved in \cite{BKM} is a ``square'' counterpart of the characterization Theorem~\ref{l1c}.

\begin{lemma}\label{l1b} If in the diagram $\mathcal D$ the space $X$ is metrizable and the map $p_Y$ is surjective, then the square $\mathcal D$ is skeletal if and only if $\mathcal D$ is open at a dense $G_\delta$-subset of $X$.
\end{lemma}

In order to formulate the spectral characterization of skeletal maps, we need to recall some information about inverse spectra from \cite[\S2.5]{En} and \cite[Ch.1]{Chi}.

For an inverse spectrum $\mathcal S=\{X_\alpha,p_\alpha^\beta,\Sigma\}$ consisting of topological spaces and continuous bonding maps, by $$\lim \mathcal S=\{(x_\alpha)_{\alpha\in \Sigma}\in\prod_{\alpha\in \Sigma}X_\alpha:\forall \alpha\le \beta \;\; p^\beta_\alpha(x_\beta)=x_\alpha\}$$we denote the limit of $\mathcal S$ and by $p_\alpha:\lim \mathcal S\to X_\alpha$, $p_\alpha:x\mapsto x_\alpha$, the limit projections.

An inverse spectrum $\mathcal S=\{X_\alpha,p_\alpha^\beta,\Sigma\}$ is called an {\em $\w$-spectrum} if
\begin{itemize}
\item each space $X_\alpha$, $\alpha\in \Sigma$, has countable weight;
%\item for every $\alpha\le\beta$ in $A$ the bonding projection $p^\beta_\alpha:X_\beta\to X_\alpha$ is surjective;
\item the index set $\Sigma$ is  {\em $\w$-complete} in the sense that each countable subset $\Sigma'\subset \Sigma$ has the smallest upper bound $\sup \Sigma'$ in $\Sigma$;
\item the spectrum $\mathcal S$ is {\em $\w$-continuous} in the sense that for any countable directed subset $\Sigma'\subset \Sigma$ with $\gamma=\sup \Sigma$ the limit map $\lim p^\gamma_\alpha:X_\gamma\to\lim \{X_\alpha,p^\beta_\alpha,\Sigma'\}$ is a homeomorphism.
\end{itemize}

Let $\mathcal S_X=\{X_\alpha,p^\beta_\alpha,\Sigma\}$ and $\mathcal S_Y=\{Y_\alpha,\pi^\beta_\alpha,\Sigma\}$ be  two inverse spectra indexed by the same directed partially ordered set $\Sigma$. A {\em morphism} $\{f_\alpha\}_{\alpha\in \Sigma}:\mathcal S_X\to\mathcal S_Y$ between these spectra is a family
of maps $\{f_\alpha:X_\alpha\to Y_\alpha\}_{\alpha\in \Sigma}$ such that $f_\alpha\circ p^\beta_\alpha=\pi^\beta_\alpha\circ f_\beta$ for any elements $\alpha\le\beta$ in $\Sigma$.

Each morphism $\{f_\alpha\}_{\alpha\in \Sigma}:\mathcal S_X\to\mathcal S_Y$ between inverse spectra induces the limit map
$$\lim f_\alpha:\lim\mathcal S_X\to\lim \mathcal S_Y,\;\;\lim f_\alpha:(x_\alpha)_{\alpha\in \Sigma}\mapsto (f_\alpha(x_\alpha))_{\alpha\in \Sigma }$$ between the limit spaces of these spectra.

Following \cite{BKM} we say that a morphism $\{f_\alpha\}_{\alpha\in \Sigma}:\mathcal S_X\to \mathcal S_Y$ between two inverse spectra $\mathcal S_X=\{X_\alpha,p_\alpha^\beta,\Sigma\}$ and $\mathcal S_Y=\{Y_\alpha,\pi_\alpha^\beta,\Sigma\}$:
\begin{itemize}
\item {\em is skeletal} if each map $f_\alpha:X_\alpha\to Y_\alpha$, $\alpha\in \Sigma$, is skeletal;
\item {\em has skeletal limit squares} if for every $\alpha\in \Sigma$ the commutative square
$$\xymatrix{
\invlim \mathcal S_X\ar[rr]^{\invlim f_\alpha}\ar[d]_{p_\alpha}&&\invlim \mathcal S_Y\ar[d]^{\pi_\alpha}\\
X_\alpha\ar[rr]_{f_\alpha}&&Y_\alpha}
$$is skeletal.
\end{itemize}

We say that two maps $f:X\to Y$ and $f':X'\to Y'$ are {\em homeomorphic} if there are homeomorphisms $h_X:X\to X'$ and $h_Y:Y\to Y'$ such that $f'\circ h_X=h_Y\circ f$.

The following spectral characterizations of skeletal maps was proved in \cite{BKM}.

\begin{theorem}\label{skel-char-comp} For a map $f:X\to Y$ between compact Hausdorff spaces the following conditions are equivalent:
\begin{enumerate}
\item $f$ is skeletal.
\item $f$ is homeomorphic to the limit map $\invlim f_\alpha:\invlim \mathcal S_X\to\mathcal S_Y$ of a skeletal morphism $\{f_\alpha\}:\mathcal S_X\to\mathcal S_Y$ between two $\w$-spectra $\mathcal S_X=\{X_\alpha,p_\alpha^\beta,\Sigma\}$ and $\mathcal S_Y=\{Y_\alpha,\pi_\alpha^\beta,\Sigma\}$ with surjective limit projections.
\item $f$ is homeomorphic to the limit map $\invlim f_\alpha:\invlim \mathcal S_X\to\mathcal S_Y$ of a morphism $\{f_\alpha\}:\mathcal S_X\to\mathcal S_Y$ with skeletal limit squares between two $\w$-spectra $\mathcal S_X=\{X_\alpha,p_\alpha^\beta,\Sigma\}$ and $\mathcal S_Y=\{Y_\alpha,\pi_\alpha^\beta,\Sigma\}$ with surjective limit projections.
\end{enumerate}
\end{theorem}

\section{Some properties of densely open squares}

In this section we assume that $\mathcal D$ is a commutative square
$$\xymatrix{
\tilde X\ar[r]^{\tilde f}\ar[d]_{p_X}&\tilde Y\ar[d]^{p_Y}\\
X\ar[r]_{f}&Y}
$$consisting of surjective maps between compact spaces.

By $$D_f=\{y\in Y:|f^{-1}(y)|=1\}\mbox{ and }D^f=\{x\in X:|f^{-1}(f(x))|=1\}$$we denote the {\em lower and upper degeneracy sets} of the map $f:X\to Y$, respectively.

\begin{lemma}\label{l2} The square $\mathcal D$ is open at the upper degeneracy set $D^f\subset X$ of $f$.
\end{lemma}

\begin{proof} Given an open neighborhood $U\subset X$ of a point $x\in D^f$, observe that the set $V=Y\setminus f(X\setminus U)$ is an open neighborhood of $f(x)$ such that $f^{-1}(V)\subset U$. Applying to this inclusion the surjective map $f$, we get $V\subset f(U)$. To see that $p_Y^{-1}(V)\subset \tilde f(p_X^{-1}(U))$, fix any point $\tilde y\in p_Y^{-1}(V)$ and using the surjectivity of the map $\tilde f$, find a point $\tilde x\in \tilde X$ with $\tilde f(\tilde x)=\tilde  y$. It follows that $f\circ p_X(\tilde x)=p_Y\circ \tilde f(\tilde x)\in V$ and hence $p_X(\tilde x)\in f^{-1}(V)\subset U$. Then $\tilde x\in p_X^{-1}(U)$ and $\tilde y=\tilde f(\tilde x)\in \tilde f(p_X^{-1}(U))$.
\end{proof}

\begin{lemma}\label{l3a} Assume that the square $\mathcal D$ is open at a point $a\in X$ and the space $X$ is first countable at $a$. Then
there is a closed subset $Z\subset X$ such that $a\in D^{f|Z}$, $f(Z)=Y$ and $\tilde f(p_X^{-1}(Z))=\tilde Y$.
\end{lemma}

\begin{proof} Being first countable at $a$, the space $X$ has a countable neighborhood base $(W_n)_{n\in\w}$ at $a$ such that $W_{n+1}\subset W_n\subset W_0=X$ for all $n\in\w$.
Let $U_0=X$ and $V_0=Y$. Using the fact that the square $\mathcal D$ is open at the point $a$, by induction on $n$, we can construct a sequence $(U_n)_{n=1}^\infty$ of open neighborhoods of $a$ in $X$ and a sequence $(V_n)_{n=1}^\infty$ of open neighborhoods of $f(a)$ in $Y$ such that
\begin{itemize}
\item $U_n\subset W_n\cap U_{n-1}\cap f^{-1}(V_{n-1})$,
\item $V_n\subset f(U_n)\cap V_{n-1}$ and $p_Y^{-1}(V_n)\subset \tilde f(p_X^{-1}(U_n))$
\end{itemize}
for every $n\in\IN$.

We claim that the set $$Z=\{a\}\cup\bigcup_{n\in\w}\overline U_n\setminus f^{-1}( V_{n+1})$$ is closed in $X$ and has the required properties: $a\in D^{f|Z}$, $f(Z)=Y$ and $\tilde f(p_X^{-1}(Z))=\tilde Y$.

The definition of the set $Z$ implies that it is closed in $X$ and $a\in D^{f|Z}$. To show that
 $\tilde f(p_X^{-1}(Z))=\tilde Y$, fix any point $\tilde y\in\tilde Y$. We need to find a point $\tilde x \in p_X^{-1}(Z)$ such that $\tilde f(\tilde x)=\tilde y$.

For this we consider separately two cases.

1) The image $y=p_Y(\tilde y)$ of $\tilde y$ coincides with $f(a)$. In this case for every $n\in\w$ we get $\tilde y\in p_Y^{-1}(V_n)\subset \tilde f(p_X^{-1}(U_n))$ and hence there is a point $\tilde x_n\in p^{-1}_X(U_n)$ such that $\tilde y=\tilde f(\tilde x_n)$. By the compactness of $\tilde X$, the sequence $(\tilde x_n)_{n\in\w}$ has an accumulation point $\tilde x\in p^{-1}_X(a)\subset p_X^{-1}(Z)$. The  continuity of the map $\tilde f$ guarantees that $\tilde f(\tilde x)=\tilde y$.

2) The point $y=p_Y(\tilde y)$ is not equal to $f(a)$. Since $V_0=Y$ and $\bigcap_{n\in\w}f(U_n)=\bigcap_{n\in\w}V_n=\{f(a)\}$, there is a unique number $n\in\w$ such that $y\in V_n\setminus V_{n+1}$. Then $\tilde y\in p_Y^{-1}(V_n)\subset \tilde f(p_X^{-1}(U_n))$ and hence there is a point $\tilde x\in p_X^{-1}(U_n)$ such that $\tilde f(\tilde x)=\tilde y$. Consider the image $x=p_X(\tilde x)\in U_n$ and observe that $f(x)=f\circ p_X(\tilde x)=p_Y\circ\tilde f(\tilde x)=p_Y(\tilde y)=y\notin V_{n+1}$. Consequently, $x\in U_n\setminus f^{-1}(V_{n+1})\subset Z$ and $\tilde x\in p_X^{-1}(x)\subset p^{-1}_X(Z)$.

Therefore   $\tilde f(p_X^{-1}(Z))=\tilde Y$. Applying to this equality the surjective map $p_Y$, we get
$$f(Z)=f\circ p_X(p_X^{-1}(Z))=p_Y\circ \tilde f(p_X^{-1}(Z))=p_Y(\tilde Y)=Y.$$
\end{proof}

\begin{lemma}\label{l3} If the square $\mathcal D$ is open at a
finite subset $A\subset X$, the space $X$ is first countable at
each point $x\in A$, and the restriction $f|A$ is injective, then
there is a closed subset $Z\subset X$ such that $f(Z)=Y$, $\tilde f(p^{-1}_X(Z))=\tilde Y$, and $A\subset D^{f|Z}$.
\end{lemma}

\begin{proof} By Lemma~\ref{l3a}, for every point $a\in A$ there is a closed subset $Z_a\subset X$ such that  $a\in D^{f|Z_a}$, $f(Z_a)=Y$, and $\tilde f(p^{-1}_X(Z_a))=\tilde Y$. Since $f|A$ is  injective, for each point $a\in A$ we can find an open neighborhood $W_a\subset Y$ of $f(a)$ in $Y$ such that the closures $\overline{W}_a$, $a\in A$, are pairwise distinct.
Let $B=Y\setminus\bigcup_{a\in A}W_a$ and $$Z=f^{-1}(B)\cup\bigcup_{a\in A}f^{-1}(\overline{W}_a)\cap Z_a.$$It is easy to check that the set $Z$ is closed and has the required properties: $A\subset D^{f|Z}$, $f(Z)=Y$ and $\tilde f(p_X^{-1}(Z))=\tilde Y$.
\end{proof}

The proof of the following simple lemma is left to the reader.

\begin{lemma}\label{l4a} A map $f:X\to Y$ is open (resp. skeletal) at a point $x\in X$
provided for some subset $Z\subset X$ that contains the point
$x$ the map $f|Z:Z\to Y$ is open (resp. skeletal) at the point $x$.
\end{lemma}

The following lemma is a ``square'' counterpart of Lemma~\ref{l4a}.

\begin{lemma}\label{l4} The square $\mathcal D$ is open (resp. skeletal) at a point $x\in X$
provided for some subset $Z\subset X$ that contains the point
$x$, and its preimage $\tilde Z=p_X^{-1}(Z)$ the square
$$\xymatrix{
\tilde Z\ar[r]^{\tilde f|\tilde Z}\ar[d]_{p_X|\tilde Z}&\tilde Y\ar[d]^{p_Y}\\
Z\ar[r]_{f|Z}&Y}
$$
is open (resp. skeletal) at the point $x$.
\end{lemma}

\section{Preliminaries on functors}

In this section we prove some auxiliary results on functors in the category $\Comp$ of compacta. From now on we assume that  $F:\Comp\to\Comp$ is a monomorphic epimorphic continuous functor.

For two compact Hausdorff spaces $X$ and $Y$ by $C(X,Y)$ we denote the space of continuous functions $f:X\to Y$, endowed with the compact-open topology. A proof of the following fact due to \v S\v cepin \cite[\S3.2]{Sh} can be found in \cite[2.2.3]{TZ}.

\begin{lemma}\label{f-l1} For any compacta $X,Y$ the map$$F:C(X,Y)\to C(FX,FY),\;\;F:f\mapsto Ff,$$ is continuous.
\end{lemma}

Next, we discuss the notion of support.
Let $X$ be a compact Hausdorff space, and $a\in FX$. We say that a point $a\in FX$ has {\em finite support} if $a\in FA$ for some finite subspace $A\subset X$. In this case we define $\supp(a)$ as the intersection $$\supp(a)=\cap\{A:a\in FA,\;A\subset X\mbox{ is finite}\}.$$
We shall often use the following fact proved in \cite{BMZ}:

\begin{lemma}\label{l:BMZ} Let $a\in FX$ be an element with finite support. If $\supp(a)\ne\emptyset$, then $a\in F(\supp(a))$. If $\supp(a)=\emptyset$, then $a\in FA$ for any non-empty closed subspace $A\subset X$.
\end{lemma}

The set of all elements with finite support in $FX$ will be denoted by $F_\w(X)$. The following lemma was proved in \cite[2.2.1]{TZ}.

\begin{lemma}\label{l:Fomega} The subset $F_\w(X)$ is dense in $FX$.
\end{lemma}

For a topological space $Y$ by $\dot Y$ we shall denote the set of isolated points of $Y$.
For a surjective function $f:X\to Y$ let
$$N_f=\{y\in Y:|f^{-1}(y)|>1\}=Y\setminus D_f\mbox{ and } N^f=\{x\in X:|f^{-1}(f(x))|>1\}=X\setminus D^f$$be the {\em lower and upper non-degeneracy sets} of $f$, respectively.

\begin{lemma}\label{f-l2} For any skeletal map $f:X\to Y$ between compacta and any dense subset $A\subset X$, the set
$$\mathcal A_0=\{a\in F_\w(X):\supp(a)\subset A,\;N_{f|\supp(a)}\subset\dot Y\}$$
is dense in $FX$.
\end{lemma}

\begin{proof} By Lemma~\ref{l:Fomega}, the set $F_\w(X)$ is dense in $FX$. So, it suffices to check that $\mathcal A_0$ is dense in $F_\w(X)$. Fix any element $a\in F_\w(X)$ and a neighborhood $O_a\subset F_\w(X)$. We need to find an element $b\in O_a\cap \mathcal A_0$.
If $\supp(a)=\emptyset$, then $a\in \mathcal A_0\cap O_a$ by the definition of $\mathcal A_0$. So we assume that $B=\supp(a)$ is not empty. By Lemma~\ref{l:BMZ}, $a\in FB$.
 By Lemma~\ref{f-l1}, the map
$$F:C(B,X)\to C(FB,FX),\;\;F:g\mapsto Fg$$is continuous and so is the map
$$F_a:C(B,X)\to FX,\;\;F_a:g\mapsto Fg(a).$$
It follows from the continuity of $F_a$ that the identity inclusion $i_B:B\to X$ has a neighborhood $O(i_B)$ in the function space $C(B,X)$ such that $F_a(g)=Fg(a)\in O_a$ for any map $g\in O(i_B)$.

We claim that there is a map $g\in O(i_B)$ such that $N_{f\circ g|B}\subset \dot Y$.
Since the compact-open topology on $C(B,X)$ coincides with the topology of pointwise convergence, for each point $x\in B$ we can find a neighborhood $O_x\subset X$ such that a map $g:B\to X$ belongs to the neighborhood $O(i_B)$ provided $g(x)\in O_x$ for all $x\in B$. Let
$C=B\cap f^{-1}(\dot Y)$.

We claim that for each point $x\in B\setminus C$ the set $f(O_x)$ is infinite. Assuming the converse, we can find a smaller neighborhood $U_x$ of $x$ such that $f(U_x)$ coincides with the singleton $\{f(x)\}$ which is open in $Y$ because of the skeletal property of $f$. In this case $f(x)\in \dot Y$ and $x\in C$, which contradicts the choice of $x$.

Let $B\setminus C=\{x_1,\dots,x_n\}$ be an enumeration of the set $B\setminus C$.
By finite induction for every $i\le n$ choose a point $x_i'\in O_{x_i}\cap A$ such that $f(x_i')\notin f(C)\cup\{f(x_j'):j<i\}$. As $f(O_{x_i})$ is infinite and $A$ is dense in $X$, the choice of $x_i'$ is always possible. After completing the inductive construction, define a map $g:B\to X$ letting  $g(x_i)=x_i'$ for $i\le n$ and $g(x)\in O_x\cap A\cap f^{-1}(f(x))$ for any $x\in C$. By the construction, $g\in O(i_B)$ and the map $f\circ g|B\setminus C$ is injective, which means that $N_{f\circ g|B}\subset f(C)\subset\dot Y$.

By the choice of the neighborhood $O(i_B)$, the element $b=Fg(a)$ lies in the neighborhood $O_a$. Since $b\in F(g(B))$, we get $\supp(b)\subset g(B)\subset A$, witnessing that $b\in\mathcal A_0$.
\end{proof}

\begin{lemma}\label{f-l3} Let $f:X\to Y$ be a skeletal map between compact Hausdorff spaces. If $\dot Y\subset D_f$, then for every dense subset $A\subset X$ the set
$$\mathcal A_1=\{a\in F_\w(X):\supp(a)\subset A,\;f|\supp(a) \mbox{ is 1-to-1}\}$$is dense in $FX$.
\end{lemma}

\begin{proof} By Lemma~\ref{f-l2}, the set $$\mathcal A_0=\{a\in F_\w(X):\supp(a)\subset A,\;N_{f|\supp(a)}\subset\dot Y\}$$ is dense in $FX$. Observe that for each $a\in\mathcal A_0$ we get
$N_{f|\supp(a)}\subset f(\supp(a))\cap \dot Y\subset f(\supp(a))\cap D_f\subset D_{f|\supp(a)}$, which implies $N_{f|\supp(a)}=\emptyset$ and  $a\in\mathcal A_1$. Now we see that the density of the set $\mathcal A_0$ implies the density of the set $\mathcal A_1\supset\mathcal A_0$ in $FX$.
\end{proof}

\section{1-Mec functors and densely open squares}

In this section we assume that $F:\Comp\to\Comp$ is a 1-mec functor and study its action on densely open squares. Let $\mathcal D$ be a commutative square
 $$\xymatrix{
\tilde X\ar[r]^{\tilde f}\ar[d]_{p_X}&\tilde Y\ar[d]^{p_Y}\\
X\ar[r]_{f}&Y
}$$ consisting of surjective maps between compact Hausdorff spaces. Applying the functor $F$ to this square, we obtain the commutative square $F\mathcal D$:
 $$\xymatrix{
F\tilde X\ar[r]^{F\tilde f}\ar[d]_{Fp_X}&F\tilde Y\ar[d]^{Fp_Y}\\
FX\ar[r]_{Ff}&FY.
}$$

\begin{lemma}\label{fo-l1} If the space $X$ is first countable and the square $\mathcal D$ is open at a non-empty subset $A\subset X$, then the square $F\mathcal D$ is open at the subset
$$\mathcal A_1=\{a\in F_\w(X):\supp(a)\subset A,\;f|\supp(a) \mbox{ is 1-to-1}\}\subset FX.$$
If $\dot Y\subset D_f$ and the set $A$ is dense in $X$, then the set $\mathcal A_1$ is dense in $FX$ and hence the square $F\mathcal D$ is densely open.
\end{lemma}

\begin{proof} Fix any point $b\in \mathcal A_1$ and consider its support $\supp(b)$. If it is not empty, put $B=\supp(b)$. If $\supp(b)$ is empty, put $B=\{z\}\subset A$ be any singleton in $A$. In both cases we have that $B\subset A$, $f|B$ is injective, and $b\in FB$, see Lemma~\ref{l:BMZ}.
Let $C=f(B)$ and observe that $f|B:B\to C$ is a homeomorphism. By
Lemma~\ref{l3}, there is a closed subset $Z\subset X$ such that
$B\subset D^{f|Z}$, $f(Z)=Y$ and $\tilde f(p_X^{-1}(Z))=\tilde Y$.

Let $\tilde Z=p_X^{-1}(Z)$, $p_Z=p_X|\tilde Z$, $f_Z=f|Z$, $\tilde f_Z=\tilde f|\tilde Z$ and consider the commutative square $\mathcal D_Z$:
 $$\xymatrix{
\tilde Z\ar[r]^{\tilde f_Z}\ar[d]_{p_Z}&\tilde Y\ar[d]^{p_Y}\\
Z\ar[r]_{f_Z}&Y
}$$that consists of surjective maps.  Applying to this square the epimorphic functor $F$, we obtain the commutative square $F\mathcal D_Z$:
 $$\xymatrix{
F\tilde Z\ar[r]^{F\tilde f_Z}\ar[d]_{Fp_Z}&F\tilde Y\ar[d]^{Fp_Y}\\
FZ\ar[r]_{Ff_Z}&FY,
}$$also consisting of surjective maps.

Taking into account that $B\subset D^{f_Z}$ and $F$ preserves finite 1-preimages, we conclude that $FB\subset D^{Ff_Z}$. By Lemma~\ref{l2}, the square $F\mathcal D_Z$ is open at $FB$. Applying Lemma~\ref{l4}, we conclude that the square $F\mathcal D$ is open at $FB$. In particular, $F\mathcal D$ is open at the point $b\in FB$.

If $\dot Y\subset D_f$, then by Lemma~\ref{f-l3}, the set $\mathcal A_1$ is dense in $FX$ and hence the square $F\mathcal D$ is densely open.
\end{proof}

\section{1-Mec functors and skeletal maps}\label{s:fsm}

In this section we study the action of 1-mec functors on some special types of skeletal maps.
As in the preceding section, $F:\Comp\to\Comp$ is a 1-mec functor in the category of compact Hausdorff spaces. Our principal result is the following theorem.

\begin{theorem}\label{t5n} For any surjective skeletal map $f:X\to Y$ between compact Hausdorff spaces the map $Ff:FX\to FY$ is skeletal at the subset
$$\mathcal A_1=\{a\in F_\w(X):f|\supp(a) \mbox{ is 1-to-1}\}.$$
If $\dot Y\subset D_f$, then the set $\mathcal A_1$ is dense in $FX$ and hence the map $Ff$ is skeletal.
\end{theorem}

\begin{proof} By Theorem~\ref{skel-char-comp}, the skeletal map $f:X\to Y$ can be identified with the limit map $\invlim f_\alpha$ of a morphism $\vec f=\{f_\alpha\}_{\alpha\in \Sigma }:\mathcal S_X\to \mathcal S_Y$ between some $\w$-spectra $\mathcal S_X=\{X_\alpha,p_\alpha^\beta,\Sigma\}$ and $\mathcal S_Y=\{Y_\alpha,\pi_\alpha^\beta,\Sigma\}$ with surjective limit projections such that for any $\alpha\in \Sigma $ the limit square $\mathcal D_\alpha$:
$$\xymatrix{
X\ar[r]^{f}\ar[d]_{p_\alpha}&Y\ar[d]^{\pi_\alpha}\\
X_\alpha\ar[r]_{f_\alpha}&Y_\alpha
}$$
is skeletal.

To show that the map $Ff$ is skeletal at each point $a\in\mathcal A_1$, fix any open neighborhood
$U\subset FX$ of $a$. We need to prove that the image $Ff(U)$ has non-empty interior in $FY$.
The inclusion $a\in\mathcal A_1$ implies that the restriction $f|\supp(a)$ is injective.

By the continuity of the functor $F$, there is an index $\alpha\in \Sigma $ and an open neighborhood $U_\alpha\subset FX_\alpha$ of the point $a_\alpha=Fp_\alpha(a)$  such that $U\supset (Fp_\alpha)^{-1}(U_\alpha)$. Replacing $\alpha$ by a larger index, if necessary, we can additionally assume that the restriction $p_\alpha|\supp(a)$ and $\pi_\alpha\circ f|\supp(a)$ are injective. Then the map $f_\alpha|p_\alpha(\supp(a))$ also is injective. Since $\supp(a_\alpha)\subset p_\alpha(\supp(a))$, the restriction $f_\alpha|\supp(a_\alpha)$ is injective.

By our assumption the limit square $\mathcal D_\alpha$ is skeletal and by Lemma~\ref{l1b} it is open at some dense subset $A_\alpha\subset X_\alpha$. Repeating the argument from the proof of Lemma~\ref{f-l2}, we can approximate the element $a_\alpha$ by an element $a_\alpha'\in U_\alpha$ such that $\supp(a_\alpha')\subset A_\alpha$ and the map $f_\alpha|\supp(a_\alpha')$ is injective. By Lemma~\ref{fo-l1}, the square $F\mathcal D_\alpha$ is open at the point $a'_\alpha$.
Then for the neighborhood $U_\alpha$ of $a_\alpha'$ there is a non-empty open set $V_\alpha\subset FY_\alpha$ such that $V_\alpha\subset Ff_\alpha(U_\alpha)$ and the open subset $V=(F\pi_\alpha)^{-1}(V_\alpha)$ of the space $FY$ lies in the image $Ff((Fp_X)^{-1}(U_\alpha))\subset Ff(U)$, which completes the proof of the skeletality of $Ff$ at $a$.

If $\dot Y\subset D_f$, then the set $\mathcal A_1$ is dense in $FX$ by Lemma~\ref{f-l3} and hence the map $Ff:FX\to FY$ is skeletal.
\end{proof}

A map $f:X\to Y$ between topological spaces is called {\em irreducible} if $f(X)=Y$ but $f(Z)\ne Y$ for each closed subset $Z\subset X$.

\begin{corollary}\label{fo-l3} For each irreducible map $f:X\to Y$ between compact Hausdorff spaces the map $Ff:FX\to FY$ is skeletal.
\end{corollary}

\begin{proof} This lemma follows from Theorem~\ref{t5n} because each closed irreducible map $f:X\to Y$ is skeletal and has $\dot Y\subset D_f$.
\end{proof}

\section{Preimage preserving functors and skeletal maps}

The following theorem implies that for normal functors $F$ the skeletality of a map $f:X\to Y$ between compacta follows from the skeletality of the map $Ff$.

\begin{theorem}\label{preimage} Let $F:\Comp\to\Comp$ be a preimage preserving mec-functor such that $F1\ne F2$. A surjective map $f:X\to Y$ between compact Hausdorff spaces is skeletal if the map $Ff:FX\to FY$ is skeletal.
\end{theorem}

\begin{proof} Assume that the map $Ff$ is skeletal. To show that $f:X\to Y$ is skeletal, fix a nowhere dense subset $N\subset Y$. We need to show that its preimage $f^{-1}(N)$ is nowhere dense in $X$. Assume conversely that $f^{-1}(N)$ contains some non-empty open set $U$. The set $F(X\setminus U)$ is closed in $FX$ and its complement $\U=FX\setminus F(X\setminus U)$ is open in $FX$. Let us show that the set $\U$ is not empty. Fix any point $u\in U$ and consider the closed subspace $Z=(X\setminus U)\cup \{u\}$ of $X$ and the continuous map $p:Z\to 2=\{0,1\}$ such that $p^{-1}(0)=X\setminus U$ and $p^{-1}(1)=\{u\}$.
By our hypothesis, $F1\ne F2$. So we can find an element $b'\in F2\setminus F1$. Since the functor $F$ is epimorphic, there is an element $a'\in FZ$ such that $Fp(a')=b'$. The element $a'$ does not belong to $F(X\setminus U)$, which implies that the sets $FZ\setminus F(X\setminus U)$ and $\U=FX\setminus F(X\setminus U)$ are not empty.

Since the map $Ff:FX\to FY$ is skeletal, the image $Ff(\U)$ of the non-empty open set $\U\subset FX$ has non-empty interior in $FY$ and hence contains some non-empty open subset $\V\subset Ff(\U)$ of the space $FY$. Since the set $N$ is nowhere dense in $Y$, the subspace $F_\w(X\setminus N)=\{a\in F_\w(X):\supp(a)\subset X\setminus N\}$ is dense in $FX$. So, we can find an element $b\in F_\w(X)\cap\V$ with finite support $\supp(b)\subset Y\setminus N$. Let $A=\supp(b)$ if $\supp(b)\ne \emptyset$ and $A=\{y\}\subset Y\setminus N$ be any singleton in $Y\setminus N$ if $\supp(b)=\emptyset$. By \cite{BMZ}, $b\in F(A)\subset F(X\setminus N)$.
Since $b\in\V\subset Ff(\U)$, there is an element $a\in\U=FX\setminus F(X\setminus U)$ with $Ff(a)=b$.
Observe that $f^{-1}(A)\subset f^{-1}(Y\setminus N)=X\setminus f^{-1}(N)\subset X\setminus U$.
Since the functor $F$ preserves preimages, we conclude that
$a\in F(f^{-1}(A))\subset F(X\setminus U)$, which contradicts the choice of $a$. This contradiction shows that the set $f^{-1}(N)$ is nowhere dense in $X$ and hence the map $f$ is skeletal.
\end{proof}

\section{Proof of Theorem~\ref{s-map}}\label{s:s-map}

To prove the ``1-mec'' part of Theorem~\ref{s-map}, assume that $F:\Comp\to\Comp$ is a 1-mec functor such that for each compact zero-dimensional space $Z$ the map $F{\mathscr 2}_Z:F(Z\oplus 2)\to F(Z\oplus 1)$ is skeletal.

\begin{lemma}\label{pm-l0} For any surjective map $f:A\to B$ between finite discrete spaces and any compact zero-dimensional space $Z$ the map $F(\id_Z\oplus f):F(Z\oplus A)\to F(Z\oplus B)$ is skeletal.
\end{lemma}

\begin{proof} Let $n=|A|\setminus |B|$ and $(A_i)_{i=0}^n$ be an increasing sequence of subsets of $A$ such that $|A_0|=|B|$, $f(A_0)=B$, $A_n=A$ and $|A_{i+1}\setminus A_i|=1$ for every $i<n$. For every positive number  $i\le n$ choose a surjective map $f_i:A_{i}\to A_{i-1}$ such that $f\circ f_i=f|A_{i}$. Observe that
$$\id_Z\oplus f=(\id_Z\oplus f_1)\circ\cdots\circ(\id_Z\oplus f_n)$$ and for every $i\le n$ the map $\id_Z\oplus f_i$ is homeomorphic to the map $\mathscr 2_{Z_i}$ where $Z_i=Z\oplus (A_{i-1}\setminus D_{f_i})$. By our assumption the map $F(\mathscr 2_{Z_i})$ is skeletal and so is its homeomorphic copy $F(\id_Z\oplus f_i)$. Since the composition of skeletal maps between compacta is skeletal, the map
$$F(\id_Z\oplus f)=F(\id_Z\oplus f_1)\circ\cdots\circ F(\id_Z\oplus f_n)$$ is skeletal.
\end{proof}

\begin{lemma}\label{pm-l1} For any surjective map $f:A\to B$ between finite discrete spaces and any compact space $X$ the map $F(\id_X\oplus f):F(X\oplus A)\to F(X\oplus B)$ is skeletal.
\end{lemma}

\begin{proof}  By \cite[3.2.2, 3.1.C]{En}, the compact space $X$ is the image of a compact zero-dimensional space $Z$ under an irreducible map $\xi:Z\to X$. Applying to the commutative diagram
$$\xymatrix{
X\oplus A\ar[rr]^{\id_X\oplus f}&&X\oplus B\\
Z\oplus A\ar[rr]_{\id_Z\oplus f}\ar[u]^{\xi\oplus\id_A}&&Z\oplus B\ar[u]_{\xi\oplus\id_B}
}$$the functor $F$, we obtain the commutative diagram
$$\xymatrix{
F(X\oplus A)\ar[rr]^{F(\id_X\oplus f)}&&F(X\oplus B)\\
F(Z\oplus A)\ar[rr]_{F(\id_Z\oplus f)}\ar[u]^{F(\xi\oplus\id_A)}&&F(Z\oplus B)\ar[u]_{F(\xi\oplus\id_B)}
}$$in which the map $F(\xi\oplus \id_A)$ is surjective, $F(\id_Z\oplus f)$ is skeletal by Lemma~\ref{pm-l0} and $F(\xi\oplus\id_B)$ is skeletal by Corollary~\ref{fo-l3}. Because of that the map $F(\id_X\oplus f)$ is skeletal.
\end{proof}

The following lemma yields the ``1-mec'' part of Theorem~\ref{s-map}.

\begin{lemma}\label{pm-l2} For any skeletal surjection $f:X\to Y$ between compacta the map $Ff:FX\to FY$ is skeletal.
\end{lemma}

\begin{proof} By Lemma~\ref{f-l2}, the set $$\mathcal A_0=\{a\in F_\w(X):N_{f|\supp(a)}\subset \dot Y\}$$ is dense in $FX$. So, the skeletality of the map $Ff$ will follow as soon as we check its skeletality at each point $a\in\mathcal A_0$. If $f|\supp(a)$ is injective, then $Ff$ is skeletal at $a$ by Theorem~\ref{t5n}.
So, we assume that $f|\supp(a)$ is not injective. In this case the support $A=\supp(a)$ is not empty and $a\in FA$ by Lemma~\ref{l:BMZ}. By our assumption, $N_{f|A}\subset\dot Y$ and hence the complement $Y\setminus N_{f|A}$ is an open-and-closed subset of $Y$. Consider the closed subspace $\tilde X=f^{-1}(Y\setminus N_{f|A})\cup N^{f|A}$ of $X$ and the topological sum $\tilde Y=(Y\setminus N_{f|A})\oplus N^{f|A}$.
Next, consider the commutative diagram
$$\xymatrix{
X\ar[r]^{f}&Y\\
\tilde X\ar[r]_{\tilde f}\ar[u]^{i}&\tilde Y\ar[u]_{h}
}$$where $i:\tilde X\to X$ is the embedding, $\tilde f$ is defined by $\tilde f|\tilde X\setminus N^{f|A}=f|\tilde X\setminus N^{f|A}$ and $\tilde f|N^{f|A}=\id$ while $h:\tilde Y\to Y$ is defined by $h|Y\setminus N_{f|A}=\id$ and $h|N^{f|A}=f|N^{f|A}$.

Applying the functor $F$ to this diagram we get the commutative diagram
$$\xymatrix{
FX\ar[r]^{Ff}&FY\\
F\tilde X\ar[r]_{F\tilde f}\ar[u]^{Fi}&F\tilde Y\ar[u]_{Fh}
}$$
Since $\tilde f$ is skeletal and the restriction $\tilde f|A$ is injective, the map $F\tilde f:F\tilde X\to F\tilde Y$ is skeletal at $a$ by Theorem~\ref{t5n}. By Lemma~\ref{pm-l1}, the map $Fh$ is skeletal. Consequently, the composition $Fh\circ F\tilde f$ is skeletal at $a$ and then $Ff$ is skeletal at $a$ by Lemma~\ref{l4a}.
\end{proof}

To prove the ``1-mecw'' part of Theorem~\ref{s-map}, assume that $F:\Comp\to\Comp$ is a 1-mecw functor such that for each zero-dimensional compact metrizable space $Z$ the map $F{\mathscr 2}_Z:F(Z\oplus 2)\to F(Z\oplus 1)$ is skeletal.
The skeletality of the functor $F$ will follow from the ``1-mec'' part of Theorem~\ref{s-map} as soon as we check that for each zero-dimensional compact space $Z$ the map $F{\mathscr 2}_Z:F(Z\oplus 2)\to F(Z\oplus 1)$ is skeletal. For this we shall apply the Characterization Theorem~\ref{skel-char-comp}.

By (the proof of) Proposition~1.3.5 of \cite{Chi}, the zero-dimensional space $Z$ is homeomorphic to the limit $\invlim \mathcal S_Z$ of an $\w$-spectrum $\mathcal S_Z=\{Z_\alpha,p_\alpha^\beta,\Sigma\}$ with surjective limit projections, consisting of zero-dimensional compact metrizable spaces $Z_\alpha$, $\alpha\in \Sigma$. For $n\in\{1,2\}$, consider the inverse spectrum  $\mathcal S_Z\oplus n=\{Z_\alpha\oplus n,p_\alpha^\beta\oplus\id_n,\Sigma\}$, where $\id_n:n\to n$ denotes the identity map of the discrete space $n=\{0,\dots,n-1\}$. Next, consider the skeletal morphism $\{\mathscr 2_{Z_\alpha}\}_{\alpha\in A}:\mathcal S_Z\oplus 2\to\mathcal S_Z\oplus 1$. Applying to this morphism the mecw functor $F$, we obtain a morphism $\{F\mathscr 2_{Z_\alpha}\}_{\alpha\in \Sigma}:F(\mathcal S_Z\oplus 2)\to F(\mathcal S_Z\oplus 1)$. By our assumption, for every $\alpha\in \Sigma$ the map $F\mathscr 2_{Z_\alpha}:F(Z_\alpha\oplus 2)\to F(Z_\alpha\oplus 1)$ is skeletal. Then Theorem~\ref{skel-char-comp} guarantees that the limit map $\invlim F\mathscr 2_{Z_\alpha}:\invlim F(\mathcal S_Z\oplus 2)\to \invlim F(\mathcal S_Z\oplus 1)$ of the skeletal morphism $\{F\mathscr 2_{Z_\alpha}\}_{\alpha\in\Sigma}$ is skeletal. By the continuity of the functor $F$, this map is homeomorphic to the map $F\mathscr 2_Z:F(Z\oplus 2)\to F(Z\oplus 1)$.

\section{Proof of Theorem~\ref{t1.5n}}

We need to prove that a 1-mecw functor $F:\Comp\to\Comp$ is skeletal if it is finitely bicommutative and finitely skeletal. By Theorem~\ref{s-map}, it suffices to check that for any zero-dimensional compact metrizable space $Z$ the map $F\mathscr 2_Z:F(Z\oplus 2)\to F(Z\oplus 1)$ is skeletal. Write the space  $Z$ as the limit of an inverse spectrum $\mathcal S_Z=\{Z_n,p_n^m,\w\}$ consisting of finite spaces $Z_n$, $n\in\w$, and surjective bonding maps $p_n^m:Z_m\to Z_n$, $n\le m$. Then the map $\mathscr 2_Z:Z\oplus 2\to Z\oplus 1$ can be identified with the limit map of the morphism $\{\mathscr 2_{Z_n}\}_{n\in\w}:\mathcal S_Z\oplus 2\to\mathcal S_Z\oplus 1$ between the inverse spectra $\mathcal S_Z\oplus 2=\{Z_n\oplus 2,p_n^m\oplus \id_2,\w\}$ and $\mathcal S_Z\oplus 1=\{Z_n\oplus 1,p_n^m\oplus \id_1,\w\}$.
Applying to this morphism the continuous functor $F$, we obtain a morphism $\{F\mathscr 2_{Z_n}\}_{n\in\w}:F(\mathcal S_Z\oplus 2)\to F(\mathcal S_Z\oplus 1)$ between the inverse spectra
$F(\mathcal S_Z\oplus 2)=\{F(Z_n\oplus 2),F(p_n^m\oplus \id_2),\w\}$ and $F(\mathcal S_Z\oplus 1)=\{F(Z_n\oplus 1),F(p_n^m\oplus \id_1),\w\}$.

The finite skeletality of the functor $F$ implies that the morphism $\{F\mathscr 2_{Z_n}\}_{n\in\w}$ consists of skeletal maps $F\mathscr 2_{Z_n}:F(Z_n\oplus 2)\to F(Z_n\oplus 1)$ for all $n\in\w$.

It is clear that for any $n\le m$ the bonding $\downarrow_n^m$-square $\mathcal D_n^m$
$$\xymatrix{
Z_m\oplus 2\ar_{p_n^m\oplus\id_2}[d]\ar^{\mathscr 2_{Z_m}}[r]&Z_m\oplus 1\ar^{p_n^m\oplus\id_1}[d]\\
Z_n\oplus 2\ar_{\mathscr 2_{Z_n}}[r]&Z_n\oplus 1}
$$is bicommutative and consists of finite spaces. Since the functor $F$ is finitely bicommutative, the bonding $\downarrow_n^m$-square  $F\mathcal D_n^m$ of the morphism $\{F\mathscr 2_{Z_n}\}_{n\in\w}$ also is bicommutative.

By Proposition 2.5 of \cite{Sh}, the bicommutativity of the bonding $\downarrow_n^m$ squares $F\mathcal D_n^m$, $n\le m$, implies the bicommutativity of the limit $\downarrow_n$-square $F\mathcal D_n$
$$\xymatrix{
F(Z\oplus 2)\ar_{F(p_n\oplus\id_2)}[d]\ar^{F(\mathscr 2_{Z})}[r]&F(Z\oplus 1)\ar^{F(p_n\oplus\id_1)}[d]\\
F(Z_n\oplus 2)\ar_{F(\mathscr 2_{Z_n})}[r]&F(Z_n\oplus 1)}
$$
for every $n\in\w$. This fact combined with the skeletality of the map $F(\mathscr 2_{Z_n})$ implies that the limit $\downarrow_n$-square $F\mathcal D_n$ is skeletal. Now Proposition 3.1 of \cite{BKM} guarantees that the limit map $F\mathscr 2_Z:F(Z\oplus 2)\to F(Z\oplus 1)$ of the morphism $\{F\mathscr 2_{Z_n}\}_{n\in\w}:F(\mathcal S_Z\oplus 2)\to F(\mathcal S_Z\oplus 1)$ is skeletal.

\section{Proof of Theorem~\ref{s-space}}\label{s:s-space}

Let $F:\Comp\to\Comp$ be a 1-mecw functor. Given a skeletally generated compact Hausdorff space $X$, we need to prove that the space $FX$ is skeletally generated.

 Represent $X$ as the limit of a continuous $\w$-spectrum $\mathcal S=\{X_\alpha,\pi^\beta_\alpha,\Sigma\}$ with surjective limit projections $p_\alpha:X\to X_\alpha$, $\alpha\in\Sigma$.
By \cite{kp8}, the space $X$, being skeletally generated, has countable cellularity. Consequently, the set $\dot X$ of isolated points of $X$ is at most countable.

For each isolated point $x\in {\dot X}$ of $X$ we can find an index $\alpha_x\in\Sigma$ such that  $\{x\}=\pi_{\alpha_x}^{-1}(U_x)$ for some open set $U_x\subset X_{\alpha_x}$, which must coincide with the singleton of $\pi_{\alpha_x}(x)$. Then for any $\alpha\ge\sup\{\alpha_x:x\in\dot X\}$ we get $\pi_\alpha(\dot X)\subset D_{\pi_\alpha}\cap \dot X_\alpha$. Replacing the index set $\Sigma$ by its cofinal subset $\{\alpha\in \Sigma:\alpha\ge \sup\{\alpha_x:x\in\dot X\}\}$, if necessary, we can assume that  $\pi_\alpha(\dot X)\subset D_{\pi_\alpha}\cap \dot X_\alpha$ for all $\alpha\in\Sigma$.

\begin{claim}\label{ps-cl} The set $\Sigma'=\{\alpha\in \Sigma:\pi_\alpha(\dot X)=\dot X_\alpha\}$ is closed and cofinal in $\Sigma$.
\end{claim}

\begin{proof} First we prove that $\Sigma'$ is closed in $\Sigma$. Given a chain $C\subset\Sigma'$ having the supremum $\sup C$ in $\Sigma$, we need to show that $\sup C\in\Sigma'$. If $\sup C\in C$, then there is nothing to prove. So we assume that $\gamma=\sup C\notin C$. In this case by the continuity of the spectrum $\mathcal S$, the space $X_\gamma$ is the limit of the inverse subspectrum $\mathcal S|C=\{X_\alpha,\pi^\alpha_\beta,C\}$.

We need to prove that $\dot X_\gamma\subset\pi_\gamma(\dot X)$. Take any  isolated point $x\in \dot X_\gamma$. By the definition of the topology  of the inverse limit $X_\gamma=\invlim\mathcal S|C$, there is an index $\alpha\in C$ such that $\{x\}=(\pi^\gamma_\alpha)^{-1}(y)$ for some isolated point $y\in X_\alpha$.
Since $y\in \dot X_\alpha=\pi_\alpha(\dot X)$, there is a point $z\in \dot X$ with  $y=\pi_\alpha(z)$. Now consider the point $x'=\pi_\gamma(z)\in X_\gamma$ and observe that $\pi^\gamma_\alpha(x')=\pi^\gamma_\alpha(\pi_\gamma(z))=\pi_\gamma(z)=y=\pi^\gamma_\alpha(x)$ and $y\in \dot X_\alpha=\pi_\alpha(\dot X)\subset D_{\pi_\alpha}\subset D_{\pi^\gamma_\alpha},$ which implies $x=x'\in \pi_\gamma(\dot X)$.

Next, we prove that $\Sigma'$ is cofinal in $\Sigma$. Given any $\alpha_0\in\Sigma$ we need to find $\alpha\in\Sigma'$ with $\alpha\ge\alpha_0$. For any isolated point $x\in \dot X_{\alpha_0}\setminus \pi_{\alpha_0}(\dot X)$ the preimage $\pi^{-1}_{\alpha_0}(x)$ is an open subset of $X$ containing no isolated points of $X$. Since the metrizable compactum $X_{\alpha_0}$ has at most countably many isolated points, and the index set $\Sigma$ is $\w$-complete, there is an index $\alpha_1\ge \alpha_0$ such that for each isolated point $x\in \dot X_{\alpha_0}\setminus\pi_{\alpha_0}(\dot X)$ the preimage $(\pi^{\alpha_1}_{\alpha_0})^{-1}(x)$ is not a singleton, which means that $\dot X_{\alpha_0}\setminus\pi_{\alpha_0}(\dot X)\subset N_{\pi^{\alpha_1}_{\alpha_0}}$. Proceeding by induction, we can construct an increasing chain $(\alpha_n)_{n\in\w}$ in $\Sigma$ such that $\dot X_{\alpha_n}\setminus \pi_{\alpha_n}(\dot X)\subset N_{\pi^{\alpha_{n+1}}_{\alpha_n}}$ for all $n\in\w$. Since $\Sigma$ is $\w$-complete, the chain $(\alpha_n)_{n\in\w}$ has the smallest upper bound $\alpha_\w=\sup_{n\in\w}\alpha_n$ in $\Sigma$.
We claim that $\alpha_\w\in\Sigma'$. Given any isolated point $x\in \dot X_{\alpha_\w}$, we need to prove that $x\in \pi_{\alpha_\w}(\dot X)$.

The continuity of the spectrum $\mathcal S$ guarantees that the space $X_{\alpha_\w}$ is the limits of the inverse sequence $\{X_{\alpha_n},\pi^{\alpha_{n+1}}_{\alpha_n},\w\}$. By the definition of the topology of the inverse limit, there is a number $n\in\w$ such that $\{x\}=(\pi^{\alpha_\w}_{\alpha_n})^{-1}(U_n)$ for some open set $U_n\subset X_{\alpha_n}$ which must coincide with the singleton of the isolated point  $y=\pi^{\alpha_\w}_{\alpha_n}(x)$. We claim that $y\in\pi_{\alpha_n}(\dot X)$. In the other case the choice of the index $\alpha_{n+1}$ guarantees that the preimage $(\pi^{\alpha_{n+1}}_{\alpha_n})^{-1}(y)$ is not a singleton and then $$(\pi^{\alpha_\w}_{\alpha_{n+1}})^{-1}((\pi^{\alpha_{n+1}}_{\alpha_n})^{-1}(y))=(\pi^{\alpha_\w}_{\alpha_n})^{-1}(y)=\{x\}$$is not a singleton, which is a contradiction.
Thus $y=\pi_{\alpha_n}(z)$ for some isolated point $z\in \dot X$. Taking into account that $y\in\pi_{\alpha_n}(\dot X)\subset D_{\pi_{\alpha_n}}\subset D_{\pi^{\alpha_\w}_{\alpha_n}}$ and $\pi^{\alpha_\w}_{\alpha_n}(x)=y=\pi^{\alpha_\w}_{\alpha_n}(\pi_{\alpha_\w}(z))$, we can show that $x=\pi_{\alpha_\w}(z)\in\pi_{\alpha_\w}(\dot X)$.
\end{proof}

It follows from Claim~\ref{ps-cl} that $X$ is the limit of the $\w$-spectrum $\mathcal S=\{X_\alpha,\pi^\alpha_\beta,\Sigma'\}$ consisting of metrizable compacta and surjective skeletal bonding projections and such that $\dot X_{\alpha}\subset D_{\pi_\alpha}\subset D_{\pi^\beta_\alpha}$ for all $\alpha\le\beta$ in $\Sigma'$. By Theorem~\ref{t5n}, the latter condition guarantees that the map $F\pi^\beta_\alpha:FX_\beta\to FX_\alpha$ is skeletal. Since the functor $F$ is epimorphic, continuous and preserves weight, the space  $FX$ is skeletally generated, being the limit of the continuous $\w$-spectrum $\{FX_\alpha,F\pi^\beta_\alpha,\Sigma'\}$ with surjective skeletal bonding projections $F\pi^\beta_\alpha:FX_\beta\to FX_\alpha$.

\section{Proof of Theorem~\ref{ps-space}}\label{s:ps-space}

Assume that $F:\Comp\to\Comp$ is a preimage preserving mecw-functor with $F1\ne F2$. We need to prove that a compact Hausdorff space $X$ is skeletally generated if and only if so is the space $FX$.

If $X$ is skeletally generated, then by Theorem~\ref{s-space}, so is the space $FX$.
Now assume conversely that the space $FX$ is skeletally generated.

Write the space $X$ as the limit of an inverse $\w$-spectrum $\mathcal S=\{X_\alpha,p_\alpha^\beta,\Sigma\}$ with surjective bonding projections.
Applying to this spectrum the functor $F$, we get the inverse $\w$-spectrum $F\mathcal S=\{FX_\alpha,Fp_\alpha^\beta,\Sigma\}$. By the continuity of the functor $F$, the limit space of the spectrum $F\mathcal S$ can be identified with $FX$. The space $FX$, being skeletally generated, is the limit of an inverse $\w$-spectrum with skeletal bonding projection. By the Spectral Theorem of \v S\v cepin \cite{Sh}, \cite[1.3.4]{Chi}, we can assume that the latter spectrum coincides with the subspectrum $\mathcal S|\Sigma'=\{FX_\alpha,Fp^\beta_\alpha,\Sigma'\}$ for some $\w$-closed cofinal subset $\Sigma'$ of the index set $\Sigma$. According to Theorem~\ref{preimage}, the skeletality of the maps $Fp_\alpha^\beta$  implies the skeletality of the maps $p^\beta_\alpha$ for any $\alpha\le \beta$ in $\Sigma'$. Consequently, the space $X$ is skeletally generated, being homeomorphic to the limit space of the inverse spectrum $\mathcal S|\Sigma'=\{X_\alpha,p_\alpha^\beta,\Sigma'\}$ with surjective skeletal bonding projections.

\section{Some Examples and Open Problems}\label{s:eop}

In this section we shall present examples of skeletal and non-skeletal functors.

For a natural number $n$ and a mec functor $F:\Comp\to\Comp$ let $F_n$ be the subfunctor of $F$ assigning to each compact space $X$ the closed subspace
$$F_n(X)=\{a\in FX:\exists \xi\in C(n,X)\mbox{ such that }a\in F\xi(Fn)\}$$of $F$.
First we observe that subfuctors $F_n$ of open functors need not be skeletal.

\begin{example}\label{ex:finitary} For the open functor $P:\Comp\to\Comp$ of probability measures and every natural number $n\ge 2$ the subfunctor $P_n$ is not skeletal.
\end{example}

This can be shown applying Theorem~\ref{s-map}. The non-skeletal functors $P_n$ are not finitary.
We recall that a functor $F:\Comp\to\Comp$ is {\em finitary} if for any finite discrete space $X$ the space $FX$ is finite. A typical example of a finitary functor is the functor $\exp$ of hyperspace, see \cite[2.1.1]{TZ}. This functor is open according to \cite[2.10.11]{TZ}.

\begin{example} For every $n\ge 3$ the subfunctor $\exp_n$ of the hyperspace functor is normal and finitary but not skeletal.
\end{example}

By Corollary~\ref{op->skel}, each open 1-mec functor is skeletal. Now we present three examples showing that the reverse implication does not hold.

By Proposition 2.10.1 \cite{TZ}, a normal functor $F:\Comp\to\Comp$ with finite supports is bicommutative if and only if $F$ is finitely bicommutative.
In \cite[p.85]{TZ} A.Teleiko and M.Zarichnyi constructed an example of a finitary normal functor $F:\Comp\to\Comp$, which is finitely bicommutative but not bicommutative. Applying to this functor Theorems~\ref{t1.5n} and \ref{scepin}, we get:

\begin{example}\label{ex:TZ} There is a finitary normal functor $F:\Comp\to\Comp$ which is  finitely bicommutative and skeletal but is not bicommutative and hence not open.
\end{example}

By Proposition 2.10.1 of \cite{TZ}, the functor from Example~\ref{ex:TZ} has infinite degree. There is also a finitary {\em weakly normal} functor of finite degree, which is skeletal but not open.

By $\lambda:\Comp\to\Comp$ we denote the functor of superextension, see \cite[2.1.2]{TZ}. It is known that the functor $\lambda$ is open, finitary, weakly normal, preserves 1-preimages but fails to preserve preimages, see \cite{KR} and Propositions 2.3.2, 2.10.13 of \cite{TZ}. By \cite[2.10.19]{TZ}, for every $n\ge 3$ the subfunctor $\lambda_3$ of $\lambda$ is not open. Using the characterization Theorem~\ref{s-map}, one can easily check that the functor $\lambda_3$ is skeletal. Thus we obtain another:

\begin{example}\label{e:lambda3} The finitary weakly normal functor $\lambda_3$ is skeletal but not open.
\end{example}

The functor $\lambda_3$ is finitary and has finite degree but is not normal. Our final example is a normal functor of finite degree which is skeletal but not open.

\begin{example}\label{e:PDelta} The functor $P_3$ contains a normal subfunctor $P_\Delta$, which is skeletal but not open.
\end{example}

\begin{proof}
In the standard 2-simplex $\Delta_2=\{(\alpha,\beta,\gamma)\in[0,1]^3:\alpha+\beta+\gamma=1\}$ consider the closed subsets
$$\begin{aligned}
\Delta_0&=\big\{(\alpha,\beta,\gamma)\in\Delta^2:\max\{\alpha,\beta,\gamma\}=1\big\},\\
\Delta_1&=\big\{(\alpha,\beta,\gamma)\in\Delta^2:\min\{\alpha,\beta,\gamma\}=0,\;\; \max\{\alpha,\beta,\gamma\}\le\frac{11}{12}\big\},\\
\Delta_2&=\big\{(\alpha,\beta,\gamma)\in\Delta^2:\min\{\alpha,\beta,\gamma\}\ge\frac1{12},\;\; \max\{\alpha,\beta,\gamma\}\ge \frac34\big\},
\end{aligned}$$
and their union $\Delta=\Delta_0\cup\Delta_1\cup\Delta_2$, which looks as follows:

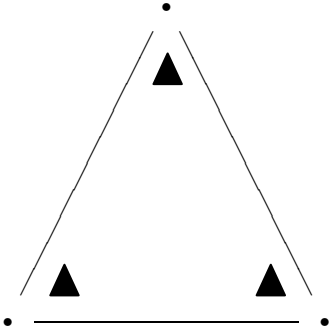
\begin{picture}(150,140)(-170,-10)
\put(0,0){\circle*{3}}
\put(5,10){\line(1,2){50}}
\put(60,119){\circle*{3}}
\put(120,0){\circle*{3}}
\put(115,10){\line(-1,2){50}}
\put(10,0){\line(1,0){100}}
\put(14,10){{\huge $\blacktriangle$}}
\put(92,10){{\huge $\blacktriangle$}}
\put(53,90){{\huge $\blacktriangle$}}
\end{picture}

Now consider the subfunctor $P_\Delta\subset P$ of the functor of probability measures assigning to each compact space $X$ the closed subspace
$$P_\Delta(X)=\{\alpha\delta_x+\beta\delta_y+\gamma\delta_z:(\alpha,\beta,\gamma)\in\Delta,\;x,y,z\in X\}\subset P(X).$$
Here $\delta_x$ stands for the Dirac measures concentrated at a point $x$. One can check that $P_\Delta$ is a normal functor of degree $\deg P_\Delta=3$. In fact, $P_\Delta$ is a subfunctor of the functor $P_3\subset P$. Theorem 2.10.21 \cite{TZ} characterizing  open normal functors of finite degree implies that the functor $P_\Delta$ is not open. Applying the characterization Theorem~\ref{s-map}, one can check that the functor $P_\Delta$ is skeletal.
%To be completed by M.K.
\end{proof}

Examples~\ref{ex:finitary}---\ref{e:PDelta} suggest the following open

\begin{problem} Assume that a finitary normal functor $F:\Comp\to\Comp$ of finite degree is skeletal. Is $F$ open? Equivalently, is $F$ finitely bicommutative?
\end{problem}

Let us also ask some other questions about skeletality of functors.

We shall say that a functor $F:\Comp\to\Comp$ is ({\em finitely}) {\em square-skeletal} if for each skeletal square $\mathcal D$ consisting of continuous surjective maps between (finite) compact spaces the square $F\mathcal D$ is skeletal.

\begin{proposition}\label{p:s2} Let $F:\Comp\to\Comp$ be an epimorphic functor.
\begin{enumerate}
\item If $F$ is (finitely) square-skeletal, then $F$ is (finitely) skeletal.
\item If $F$ is (finitely) bicommutative and (finitely) skeletal, then $F$ is (finitely) square-skeletal.
\item If $F$ is finitary, then $F$ is finitely square-skeletal if and only if $F$ is finitely bicommutative only if $F$ is skeletal.
\end{enumerate}
\end{proposition}

\begin{proof} 1,2. The first two statements follow from Remarks~\ref{rem1} and \ref{rem1a}, respectively.
\smallskip

3. The third statement follows from Theorem~\ref{t1.5n} and an observation that a commutative square consisting of epimorphisms between finite spaces is skeletal if and only if it is bicommutative.
\end{proof}

Proposition~\ref{p:s2} suggests another two problems:

\begin{problem}\label{prob:2} Is each (finitary) skeletal normal functor $F:\Comp\to\Comp$ (finitely) square-skeletal?
\end{problem}

\begin{problem} Is a normal functor $F:\Comp\to\Comp$ skeletal if it is finitely square-skeletal? {\rm (Theorem~\ref{t1.5n} implies that the answer is affirmative if the functor $F$ is finitary).}
\end{problem}

It is clear that each functor that preserves (1-)preimages preserves finite (1-)preimages. We do not know if the converse statement is true.

\begin{problem} Does a mec-functor $F:\Comp\to\Comp$ preserve (1-)preimages if $F$ preserves finite (1-)preimages.
\end{problem}

\section{Acknowledgements}

The authors would like to thank Vesko Valov for valuable comments concerning skeletally generated compacta.

\end{document}